\documentclass[]{article}
\usepackage{amssymb,calc,enumerate,booktabs}
\usepackage{multirow}
\usepackage{etoolbox,rotating,graphicx}
\usepackage{savesym}
\savesymbol{AND}
\usepackage{algorithmic}
\usepackage{algorithm}
\usepackage{caption}
\usepackage{fixltx2e,url}
\usepackage{savesym}
\usepackage{wrapfig}
\usepackage{amsmath}
\usepackage[subrefformat=parens,labelformat=parens]{subcaption}
\usepackage{comment}
\savesymbol{theoremstyle}
\usepackage{amsthm}
\usepackage{arydshln}
\DeclareGraphicsExtensions{.pdf,.jpeg,.png}
\restoresymbol{theorem}{theoremstyle}
\newtheorem{theorem}{Theorem}
\newtheorem{observation}[theorem]{Observation}

\newtheorem{lemma}[theorem]{Lemma}
\newtheorem{corollary}[theorem]{Corollary}

\newtheorem{definition}[theorem]{Definition}
\newtheorem{proposition}[theorem]{Proposition}

\newcommand{\ff}{\mathcal{F}}

\begin{document}


\title{A characterization of linearizable instances of the quadratic minimum spanning tree problem}

\author{\sc{Ante \'Custi\'c}\thanks{{\tt acustic@sfu.ca}.
Department of Mathematics, Simon Fraser University Surrey,
 250-13450 102nd AV, Surrey, British Columbia, V3T 0A3, Canada}
\and
\sc{Abraham P. Punnen}\thanks{{\tt apunnen@sfu.ca}. Department of Mathematics, Simon Fraser University Surrey,
 250-13450 102nd AV, Surrey, British Columbia, V3T 0A3, Canada}}

\maketitle

\begin{abstract}
We investigate special cases of the quadratic minimum spanning tree problem \mbox{(QMSTP)} on a graph $G=(V,E)$ that can be solved as a linear minimum spanning tree problem. Characterization of such problems on graphs with special properties are given. This include complete graphs, complete bipartite graphs, cactuses among others. Our characterization can be verified in $O(|E|^2)$ time. In the case of complete graphs and when the cost matrix is  given in factored form, we show that our characterization can be verified in $O(|E|)$ time. Related open problems are also indicated.\medskip

\end{abstract}


\section{Introduction} The minimum spanning tree problem (MSTP) is well studied in the combinatorial optimization literature. A generalization of this problem, called the {\it quadratic minimum spanning tree problem} (QMSTP), recently received considerable attention from the research community. Some of these papers focus on exact algorithms~\cite{Assad,x9} while the majority of published works deal with heuristic algorithms~\cite{x2,x8,x3,Oncan,x1,x10,Zhou}. Isolated results on some theoretical properties of the problem are also available.
 Special cases of QMSTP studied in the literature include multiplicative objective functions~\cite{y3,y4,y5}, spanning trees with conflict constraints~\cite{da,Zh}, and spanning tree problems with one quadratic term~\cite{x5,x4}. Some polynomially solvable special cases of QMSTP are discussed in~\cite{ZCP15} along with various complexity results.

Let $G=(V,E)$ be a graph such that $|V|=n$ and $E=\{1,2,\ldots ,m\}$. For each $(e,f)\in E\times E$ a cost $q(e,f)$ is prescribed. Let $\ff$ be the family of all spanning trees of $G$ and $Q$ be the $m\times m$ matrix with its $(i,j)$-th element as $q(i,j)$. For each $T\in \ff$ its {\it cost} $Q(T)$ is given by $$Q(T)= \sum_{e\in T}\sum_{f\in T} q(e,f).$$  
Note that the notation $e\in T$ means $e$ belongs to the edge set of $T$. Then the QMSTP is to find a spanning tree $T$ in $\ff$ such that $Q(T)$ is as small as possible. Similarly, for each $T\in \ff$, let $C(T)=\sum_{e\in T}c(e)$, where $c(e)$ is a prescribed cost of edge $e\in E$. Given a cost matrix $Q$,  the {\it quadratic spanning tree linearization problem} (QST-LP) is to determine if there exists a linear cost vector $C=(c(1),c(2),\ldots,c(m))$ such that $Q(T)=C(T)$ for all $T\in \ff$.  If the answer to this decision problem is `yes', the quadratic cost matrix $Q$ is said to be {\it linearizable} and $C$ is called a {\it linearization} of $Q$. Note that $|\ff|$ could be as large as $n^{n-2}$ and hence QST-LP is a non-trivial problem. In fact, there is no immediate direct way to test if QST-LP belongs to NP.

The linearization problem for the quadratic assignment problem (QAP) was considered by Kabadi and Punnen~\cite{y10}, Adams and Wadell~\cite{AW}, and \c{C}ela et al.~\cite{CDW}. The special case of  Koopmans-Beckman QAP linearization problem was studied by Punnen and Kabadi~\cite{y9} and \c{C}ela et al.~\cite{CDW}. In this paper, we provide a characterization of linearizable instances of QMSTP on a wide class of graphs, including the complete graph, complete bipartite graph, cactus etc. Our characterization can be tested in $O(m^2)$ time. Also an $O(n)$ algorithm  for recognizing an $n\times n$ sum matrix 
represented in factored form is given. In the case of complete (bipartite) graphs, this leads to an $O(m)$ algorithm to test if symmetric matrix $Q$ is linearizable when 
represented in factored form. As a byproduct of these results, we have new polynomially solvable special cases of the QMSTP.

\section{Preliminaries}

QMSTP is well known to be strongly NP-hard. In fact, it is NP-hard even if $Q$ is of rank one~\cite{y7}. A special case of the rank one QMSTP is called the {\it multiplicative minimum spanning tree problem} (MMSTP) considered by various authors~\cite{y3,y4,y1,y5,Oncan,y7,z1,y2}. The MMSTP on the graph $G$ can be stated as
\begin{align*}\label{qmst_of}
\text{Minimize\hspace{9pt}}& \left(\sum_{e\in T}d^1_e+\delta_1\right )\left(\sum_{e\in T}d^2_e+\delta_2\right )\\
\text{Subject to } & T\in \ff,
\end{align*}
where $d^1_e,d^2_e$ are two prescribed weights of the edge $e\in E$ and $\delta_1,\delta_2$ are constants. If $\left(\sum_{e\in T}d^1_e+\delta_1\right ) > 0$  and $\left(\sum_{e\in T}d^2_e+\delta_2\right ) > 0$ for all $T\in\ff$, MMSTP can be solved in polynomial time using the parametric minimum spanning tree problem~\cite{y1,Oncan,y7,z1,y2}. If $d^1_e,d^2_e$ are allowed to take any real values, then MMSTP is known to be NP-hard~\cite{y7}. We now observe that MMSTP is NP-hard even if $d^1_e,d^2_e \geq 0$ but $\delta_1$ and $\delta_2$ are arbitrary.

\begin{theorem}
	The MMSTP is NP-hard even if $d^1_e,d^2_e \geq 0$ for all $e\in E$.
\end{theorem}
\begin{proof}
We reduce the {\it subset sum problem} to the MMSTP. The subset sum problem can be stated as follows. Given non-negative numbers $a_1,a_2,\ldots ,a_n$ and a constant $K$, determine if there exists a subset $S$ of $\{1,2,\ldots ,n\}$ such that $\sum_{i\in S}a_i=K$. From an instance of subset sum problem, we construct an instance of MMSTP as follows. For each $i=1,2,\ldots, n$ create a 3-cycle on nodes $v^i_1,v^i_2,v^i_3$. Link these 3-cycles using the path $v^1_1-v^2_1-\cdots -v^n_1$. For the edge $e^i_{12}=(v^i_1,v^i_2)$ assign cost $d^1_{e^i_{12}}=a_i, i = 1,2,\ldots, n$.  For any other edge $e$, we set $d^1_e=0$. Choose the vector $d^2=d^1$ and $\delta_1=\delta_2=-K$. It can be verified that the resulting MMSTP has an optimal spanning tree with objective function value zero if and only if the subset sum problem has a solution. The result now follows from the NP-completeness of the subset sum problem.
\end{proof}

We continue this section by presenting some useful basic facts about the \mbox{QMSTP}.

Let $M^{m\times m}$ be the vector space of all real valued $m\times m$ matrices. The set of linearizable quadratic cost matrices for QMSTP on a given graph with $m$ edges forms a subspace of $M^{m\times m}$. As a consequence we have the following.

\begin{observation}\label{obs1}
	Let $Q_1$ and $Q_2$ be two cost matrices for the QMSTP on a graph $G$. If $Q_1$ and $Q_2$ are linearizable, then $\alpha Q_1+\beta Q_2$ is also linearizable for any scalars $\alpha$ and $\beta$. Furthermore, if $C_1$ is a linearization of $Q_1$ and $C_2$ is a linearization of $Q_2$, then $\alpha C_1+\beta C_2$ is a linearization of $\alpha Q_1+\beta Q_2$.
\end{observation}

A square matrix $A$ is said to be a \emph{skew-symmetric matrix} if $A^T=-A$.

\begin{observation}\label{obs2}
If $Q$ is a cost matrix for the QMSTP on a graph $G$, $A$ is any skew-symmetric matrix and $D$ is a diagonal matrix, all of the same size, then $Q$ is linearizable if and only if $Q+A+D$ is linearizable.
\end{observation}

It may be noted that if $Q$ is skew-symmetric then $Q(T)=0$ for any spanning tree $T$. Thus a skew-symmetric matrix is linearizable for any graph $G$.

\begin{observation}\label{obs3}
If $Q$ is a cost matrix for the QMSTP on a graph $G$. Then $Q$ is linearizable if and only if $\frac{1}{2}(Q+Q^T)$ is linearizable. Furthermore, $C$ is a linearization of $Q$ if and only if $C$ is a linearization of $\frac{1}{2}(Q+Q^T).$
\end{observation}
\begin{proof}
Note that $Q= \frac{1}{2}(Q-Q^T)+\frac{1}{2}(Q+Q^T)$. As $ \frac{1}{2}(Q-Q^T)$ is skew-symmetric, the result follows from Observation~\ref{obs1}, Observation~\ref{obs2} and the fact that the null-vector is a linearization of a skew-symmetric matrix.
\end{proof}

As noted earlier, any skew-symmetric matrix is linearizable regardless the structure of the underlaying graph. We now observe that if the underlying graph is a cycle, then the resulting QMSTP is linearizable regardless the structure of the cost matrix $Q$.

\begin{theorem}\label{thm:cycle}
QMSTP on a cycle is linearizable for any cost matrix $Q$. Further, the linearization $C = (c(1),c(2),\ldots ,c(n)) $ is given by 
\begin{equation}\label{eq:cycle}
	c(e)=q(e,e)+\sum_{i\in E\setminus \{e\}}(q(i,e)+q(e,i))-\frac{\sum_{i\in E}\sum_{j\in E, j\neq i}q(i,j)}{n-1}.
\end{equation}
\end{theorem}
\begin{proof}
Let $G$ be a cycle with edges $e_1,e_2,\ldots ,e_n$. Then the spanning trees of $G$ are precisely $T_1,T_2,\ldots ,T_n$ where $T_i=G\setminus\{e_i\}$.  We need to find a vector $C=(c(1),c(2),\ldots ,c(n))$ such that
\[\sum_{e\in T_i}c(e)=Q(T_i)\]
for all $i=1,\ldots,n$. Equivalently, we want to find a solution to the linear system above, where the variables being $c(1),c(2),\ldots ,c(n)$. It can be verified that the coefficient matrix is invertible and hence the system has a unique solution. The formula for the linearization can be verified by simple algebra.
\end{proof}

The following is an immediate corollary of Theorem~\ref{thm:cycle}.

\begin{corollary}
The QMSTP is linearizable for any cost matrix $Q$ on the graph $T\cup \{e\}$ where $T$ is a tree and $e$ is an edge (not necessarily in $T$) joining two vertices of $T$.
\end{corollary}

Note that the result of Theorem~\ref{thm:cycle} can be extended to any real valued objective function for a spanning tree, not simply the quadratic objective function.

We conclude this section with few more definitions and observations that we make use later in this paper.

\begin{definition}
	An $n_1\times n_2$ matrix $H=(h(i,j))$ is called a \emph{sum matrix} if there exist vectors $a=(a(1),a(2),\ldots,a(n_1))$ and  $b=(b(1),b(2),\ldots ,b(n_2))$ such that $h(i,j)=a(i)+b(j)$ for all  $i=1,\ldots,n_1$, $j=1,\ldots,n_2$. A square matrix is called a \emph{weak sum matrix} if the relation above is not mandatory for the elements on the diagonal.
\end{definition}

Note that if an $n\times n$ square sum matrix $H=(h(i,j))$ is symmetric, then $h(i,j)=a(i)+a(j)$ for all $i,j=1,2,\ldots ,n$, for some vector $a=(a(1),\ldots ,a(n))$. Similarly, if an $n\times n$ square weak sum matrix $H=(h(i,j))$ is symmetric, then $h(i,j)=a(i)+a(j)$ for all $i,j=1,2,\ldots,n,\ i\neq j$, for some vector $a=(a(1),\ldots ,a(n))$.

\begin{definition}
	A maximal biconnected subgraph of a simple graph $G$ is called a \emph{biconnected component} of $G$.
\end{definition}


The following fact is straightforward to prove, for example see~\cite[Ch.\@ 5, p.\@ 101]{Thesis}.

\begin{proposition}\label{prop:mst_covp}
	An instance of the MSTP on a graph $G$ has the property that every spanning tree has the same cost, if and only if all edges from the same biconnected component of $G$ have the same cost.
\end{proposition}

Note that $\frac{1}{2}(Q+Q^T)$ is a symmetric matrix. Thus in view of Observation~\ref{obs3} hereafter we assume without loss of generality that the cost matrix $Q$ is symmetric.

\section{Characterization of Linearizable QMSTP}

In this section we investigate characterizations of linearizable QMSTP instances on various graph classes. Our findings are summarized in Theorem~\ref{thm:gen}.
We begin with a sufficient condition. 

\begin{lemma}\label{lm:sum}
	Let $Q$ be a symmetric cost matrix of the QMSTP on a graph $G=(V,E)$ such that for every pair $I$, $J$ of biconnected components of $G$, the submatrix of $Q$ defined by rows $I$ and columns $J$ is a sum matrix if $I\neq J$, or a symmetric weak sum matrix if $I=J$.  Then $Q$ is linearizable and a linearization of $Q$ can be computed in $O(|E|^2)$ time.
\end{lemma}
\begin{proof}
	Let $Q$ be a symmetric matrix that satisfies the hypothesis of the lemma. Then $Q$ can be expressed as 
\begin{equation}\label{eq:sep}
Q=A+A^T+D,
\end{equation}
where $D=(d(i,j))$ is a diagonal matrix, and matrix $A=(a(i,j))$ has the property that $a(i,j)=a(i,k)$ if $j$ and $k$ are edges from the same biconnected component. Note that matrices $A$ and $D$ can be found in $O(|E|^2)$ time. From Proposition~\ref{prop:mst_covp} it follows that an MSTP instance defined by any row of $A$ (i.e.\@ for some fixed row $i$ we define the length of an edge $j$ to be $a(i,j)$) has the property that every spanning tree has the same cost. Let $r(i)$ denotes the constant objective function value of the MSTP corresponding to the $i$-th row of $A$.
Then the objective value of the QMSTP for some spanning tree $T$ is
\begin{align*}
	Q(T)&=\sum_{e\in T}\sum_{f\in T}(a(e,f)+a(f,e)+d(e,f))\\
	&=\sum_{e\in T}\sum_{f\in T}(a(e,f)+a(f,e))+\sum_{e\in T}d(e,e)\\
	&=\sum_{e\in T}(r(e)+r(e))+\sum_{e\in T}d(e,e)\\
	&=\sum_{e\in T}(2r(e)+d(e,e)).
\end{align*}
Hence, by setting 
\begin{equation}\label{eq:lin}
	c(i):=2r(i)+d(i,i)
\end{equation}
we obtain a linearization of $Q$. Note that the choice of matrix $A$ is not always uniquely determined, hence a linearization is not necessarily unique.
\end{proof}

Next, we present characterizations of linearizable instances of QMSTP for some special types of graphs. We start with the complete graph. In this case, the linearization characterization shows that $Q$ must be a weak sum matrix, which, according to Lemma~\ref{lm:sum}, is the most restrictive possible characterization.

\begin{theorem}\label{thm:complete}
	A symmetric cost matrix $Q$ of the QMSTP on a complete graph $K_n$ is linearizable if and only if it is a symmetric weak sum matrix. Further, a linearization of a linearizable symmetric matrix $Q$ is given by \eqref{eq:lin}.
\end{theorem}
\begin{proof}
If $Q$ is a weak sum matrix, then from Lemma~\ref{lm:sum} it follows that $Q$ is linearizable and a linearization is given by \eqref{eq:lin}. Note that in the case of a complete graph $K_n$, entries of the matrix $A$ in the expression \eqref{eq:sep} are the same for every fixed row. Hence, $r(i)=(n-1)a(i,j)$ for any column $j$.

	Next we assume that $Q$ is linearizable. For $n\leq 3$ every corresponding symmetric square cost matrix is a symmetric weak sum matrix, hence we can assume that  $n\geq 4$.

Consider an ${n\choose 2}\times {n \choose 2}$ sum matrix $M=(m(i,j))$ of the form $m(i,j)=a(i)+a(j)$, where $a(1)=0$ and $a(i)=q(i,1)$ for $i=2,\ldots, {n \choose 2}$.  By subtracting $M$ and an appropriate diagonal matrix from $Q$, we could obtain zeros on the first row, the first column and the diagonal. By Lemma~\ref{lm:sum} $M$ is linearizable, and furthermore, any diagonal matrix is linearizable. Hence, from Observation~\ref{obs1} it follows that without loss of generality we can assume that elements of the first row, the first columns and the diagonal of $Q$ are equal to zero. In that case, $Q$ is a weak sum matrix if and only if all elements of $Q$ that are not in the first row, the first column or on the diagonal, have the same value.

Now we assume the contrary, i.e. there are two elements of $Q$ (not in the first row/column or on the diagonal) that have different values. Moreover, due to the symmetry of $Q$, there is a row $b$ that contains such two distinct value elements $q(b,x)$ and $q(b,y)$. As any element of row $b$ (except $q(b,1)$ and $q(b,b)$) can be a member of such pair, without loss of generality we can assume that edges 
 1 and $x$ are nonadjacent.

Next show that there exists a cycle $C_1$ that contains edges $1$ and $b$, and a cycle $C_2$ that contains edges $x$ and $y$ with the following property:  $C_1\cup C_2\setminus \{e,f\}$ does not contains a cycle for all $e\in\{1,b\}$, $f\in\{x,y\}$. Namely, in the case when there are no two pairs of edges from $\{1,b\}\times\{x,y\}$ that are adjacent, it is straightforward to construct $C_1$ and $C_2$ that are edge disjoint and satisfy the above property, see Figure~\ref{fig1}\subref{fig1a}.
\begin{figure}[h]
	 \centering
        \begin{subfigure}[b]{0.18\textwidth}
		\includegraphics[width=\textwidth]{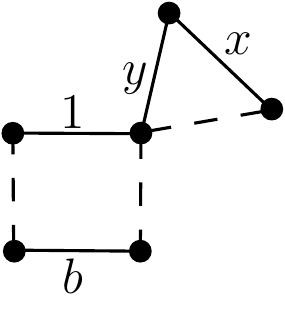}
              \caption{}
              \label{fig1a}
        \end{subfigure}\qquad
	 \begin{subfigure}[b]{0.14\textwidth}
              \includegraphics[width=\textwidth]{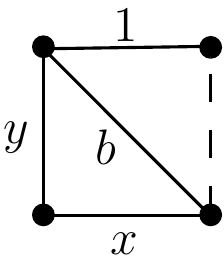}
              \caption{}
              \label{fig1b}
        \end{subfigure}\qquad
	\begin{subfigure}[b]{0.14\textwidth}
              	 \includegraphics[width=\textwidth]{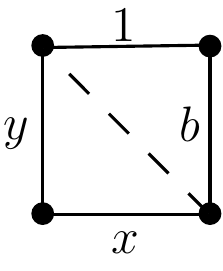}
              \caption{}
              \label{fig1c}
        \end{subfigure}\qquad
	\begin{subfigure}[b]{0.14\textwidth}
              	 \includegraphics[width=\textwidth]{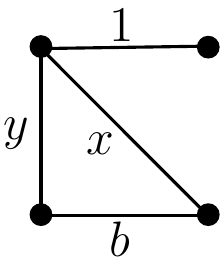}
              \caption{}
              \label{fig1d}
        \end{subfigure}
	\caption{Configurations of $\{1,b,x,y\}$ and corresponding $C_1\cup C_2$}
	\label{fig1}
\end{figure}
In the case when there are at least two pairs of edges from $\{1,b\}\times\{x,y\}$ that are adjacent, every possible $1,b,x,y$ configuration instances can be reduced to only two cases presented in Figure~\ref{fig1}\subref{fig1b} and Figure~\ref{fig1}\subref{fig1c}.
These reductions consist of symmetries (defined by exchanging sets $\{1,b\}$ and $\{x,y\}$, and by exchanging elements inside of those two sets), and edge contractions (that can induce more incidences and only make the case more complicated). These configurations in Figure~\ref{fig1}\subref{fig1b} and Figure~\ref{fig1}\subref{fig1c} are extended with a (dashed) edge that constitutes feasible $C_1$ and $C_2$. Note that we used the fact that 1 and $x$ are nonadjacent, otherwise there are instance for which $C_1$ and $C_2$ with the property above do not exist, see Figure~\ref{fig1}\subref{fig1d}.

Let $T$ be a minimum cardinality set of edges of a tree connected to both $C_1$ and $C_2$ that spans the remaining vertices. Then we define $B$ to be $T\cup C_1\cup C_2 \setminus \{1,b,x,y\}$. It is easy to see that $B$ extended by any two edges $e\in\{1,b\}$ and $f\in\{x,y\}$  forms a spanning tree.

Let $C=(c(i))$ be a cost vector that linearizes $Q$. Since both $B\cup\{b,x\}$ and $B\cup\{1,x\}$ form a spanning tree, we have that
\[
	C(B\cup\{b,x\})-C(B\cup\{1,x\})=\sum_{e\in B\cup\{b,x\}}c(e)-\sum_{e\in B\cup\{1,x\}}c(e)=c(b)-c(1).
\]
Analogously, $C(B\cup\{b,y\})-C(B\cup\{1,y\})=c(b)-c(1)$, hence
\begin{equation}\label{eq1}
	 Q(B\cup\{b,x\})-Q(B\cup\{1,x\})=Q(B\cup\{b,y\})-Q(B\cup\{1,y\}).
\end{equation}
Now let us express the cost of the spanning tree $B\cup\{b,x\}$ in terms of the quadratic cost matrix $Q$. Since $q(e,e)=0$ $\forall e,$ we have
\begin{align*}
	Q(B\cup\{b,x\})&=\sum_{e\in B\cup\{b,x\}}\sum_{f\in B\cup\{b,x\}}q(e,f)\\
	&=\sum_{e\in B}\sum_{f\in B}q(e,f)+\sum_{e\in B}2q(b,e)+\sum_{e\in B}2q(x,e)+2q(b,x).
\end{align*}
Since $q(1,e)=0$ $\forall e$ we analogously have
\[
	Q(B\cup\{1,x\})=\sum_{e\in B}\sum_{f\in B}q(e,f)+\sum_{e\in B}2q(x,e).
\]
Therefore
\begin{equation}\label{eq2}
	Q(B\cup\{b,x\})-Q(B\cup\{1,x\})=\sum_{e\in B}2q(b,e)+2q(b,x).
\end{equation}
Analogously
\begin{equation}\label{eq3}
	Q(B\cup\{b,y\})-Q(B\cup\{1,y\})=\sum_{e\in B}2q(b,e)+2q(b,y).
\end{equation}
Then from \eqref{eq1}, \eqref{eq2} and \eqref{eq3} follows that $q(b,x)=q(b,y)$ which is a contradiction to our choice of $b,x$ and $y$.
\end{proof}

Next we will generalize the approach above to obtain a characterization of linearizable cost matrices for QMSTP for a more general class of graphs (see Theorem~\ref{thm:gen}). First we present a tool which is used to prove such characterizations.

\begin{definition}
	Let $a,b,x$ and $y$ be distinct edges of a simple graph $G$ with $n$ vertices. We say that a set $B$ of $n-3$ edges is an \emph{$\{a,b\}$-$\{x,y\}$-backbone} of $G$ if adding any two edges $e\in\{a,b\}$ and $f\in\{x,y\}$ to $B$ generates a spanning tree of $G$.
\end{definition}

\begin{lemma}\label{lm:back}
	Let $Q$ be a linearizable symmetric cost matrix of the QMSTP on a simple graph $G$, and let $a$ and $x$ be two fixed distinct edges of $G$. If for all additional edges $b$ and $y$ there exist a sequence of $k\geq2$ edges $z_1,z_2,\ldots,z_k$ such that $x=z_1$, $y=z_k$ and there exists an $\{a,b\}$-$\{z_i,z_{i+1}\}$-backbone $B_i$ for every $i=1,\ldots, k-1$, then $Q$ is a symmetric weak sum matrix.
\end{lemma}
\begin{proof}
	Let $Q$ be linearizable and let $a,b,x,y$ be four distinct edges such that there exists $\{a,b\}$-$\{x,y\}$-backbone $B$. Since $Q$ is linearizable it follows that
\begin{equation}\label{back1}
	 Q(B\cup\{b,x\})-Q(B\cup\{a,x\})=Q(B\cup\{b,y\})-Q(B\cup\{a,y\}).
\end{equation}
Namely, by expressing spanning tree objective values from \eqref{back1} with a linearization costs $C=(c(i))$, one gets $c(b)-c(a)=c(b)-c(a)$. However, by expressing spanning tree objective vales from \eqref{back1} with quadratic  costs $Q=(q(i,j))$, one gets
\begin{equation*}
	q(b,x)-q(a,x)=q(b,y)-q(a,y).
\end{equation*}

Now assume that $a$ and $x$ are fixed and there exist edges $b,y$ and $z_1,\ldots,z_k$
with $z_1=x,z_k=y$, such that there exists $\{a,b\}$-$\{z_i,z_{i+1}\}$-backbone $B_i$ for every $i=1,\ldots, k-1$. Then by the same reasoning as above for all $i=1,\ldots,k-1$, we obtain the following system of equations:
\begin{align*}
	q(b,x)-q(a,x)&=q(b,z_2)-q(a,z_2),\\
	q(b,z_2)-q(a,z_2)&=q(b,z_3)-q(a,z_3),\\
		&\hspace{6pt} \vdots \\
	q(b,z_{k-1})-q(a,z_{k-1})&=q(b,y)-q(a,y).
\end{align*}
As the right-hand side of every $i$-th equation is identical to the left-hand side of $(i+1)$-th equation, it follows that $q(b,x)-q(a,x)=q(b,y)-q(a,y)$, which can be rearranged to
\begin{equation}\label{proof1}
	q(b,y)=q(b,x)+q(a,y)-q(a,x).
\end{equation}
Note that \eqref{proof1} is satisfied also for $b=a$ or $y=x$.

By the assumption of the lemma, we can obtain \eqref{proof1} for all $b$ and $y$, therefore it follows that $q(b,y)$ is a sum of a function of $b$ and a function of $y$ (as $a$ and $x$ are fixed), i.e.
\begin{equation*}
q(i,j)=s(i)+t(j)\qquad \forall i\neq j,
\end{equation*}
for some vectors $s=(s(i))$ and $t=(t(i))$.
As $Q$ is symmetric it follows that
\begin{equation*}
q(i,j)=w(i)+w(j)\qquad \forall i\neq j,
\end{equation*}
for some vector $w=(w(i))$, which proves the lemma.
\end{proof}

Note that Theorem~\ref{thm:complete} can be proven using Lemma~\ref{lm:back} and the fact that if $a$ and $x$ are two nonadjacent edges of a complete graph, then for any other pair of edges $b$ and $y$ there exists an $\{a,b\}$-$\{x,y\}$-backbone.

\begin{corollary}\label{cor:back}
	Let $Q$ be a linearizable symmetric cost matrix of the QMSTP on a simple graph $G$. Let $I$ and $J$ be two disjoint sets of edges of $G$, and let $a\in I$ and $x\in J$ be two fixed edges. Let $Q_{IJ}$ be the submatrix of $Q$ defined by rows $I$ and columns $J$. If for all additional edges $b\in I$ and $y\in J$ there exist a sequence of $k\geq2$ edges $z_1,z_2,\ldots,z_k$ such that $x=z_1$, $y=z_k$ and there exists an $\{a,b\}$-$\{z_i,z_{i+1}\}$-backbone $B_i$ for every $i=1,\ldots, k-1$, then $Q_{IJ}$ is a sum matrix.
\end{corollary}
\begin{proof}
	The proof is similar as that of Lemma~\ref{lm:back}.
\end{proof}

In the majority of cases we use Lemma~\ref{lm:back} and Corollary~\ref{cor:back}, $k$ will be equal to 2, i.e.\@ we will not need additional edges $z_i$.

Given the edges $a$, $b$, $x$ and $y$, usually we try to build an $\{a,b\}$-$\{x,y\}$-backbone in the following way. We aim to find a cycle $C_1$ that contains $a$ and $b$ and a cycle $C_2$ that contains $x$ and $y$, such that if intersection of $C_1$ and $C_2$ is nonempty, then it is connected and does not contains a pair of edges from $\{a,b\}\times\{x,y\}$.
We call such $C_1$ and $C_2$ as \emph{feasible backbone cycles} for $a,b,x$, and $y$.
For feasible backbone cycles $C_1$ and $C_2$, $(C_1\setminus \{a,b\}) \cup (C_2\setminus \{x,y\})$ extended by a tree which is connected to $C_1$ and $C_2$ and spans the remaining set of vertices, forms an $\{a,b\}$-$\{x,y\}$-backbone.

\begin{lemma}\label{lm:bipartite}
	A symmetric cost matrix $Q$ of the QMSTP on a complete bipartite graph $K_{n_1,n_2}$ with $\min\{n_1,n_2\}\geq 3$, is linearizable if and only if $Q$ is a symmetric weak sum matrix. A linearization of a lineariable symmetric matrix $Q$ is given by \eqref{eq:lin}.
\end{lemma}
\begin{proof}
	If $Q$ is a weak sum matrix, then from Lemma~\ref{lm:sum} it follows that $Q$ is linearizable and a linearization is given by \eqref{eq:lin}. Note that in the case of the complete bipartite graph $K_{n_1,n_2}$, entries of the matrix $A$ in the expression \eqref{eq:sep} are the same for every fixed row. Hence, $r(i)=(n_1+n_2-1)a(i,j)$ for any column $j$.

	 Let $\min\{n_1,n_2\}\geq 3$ and assume that $Q$ is linearizable. We fix two arbitrary nonadjacent edges $a$ and $x$. We will show that for any two additional edges $b$ and $y$, ($b\neq y$) conditions of Lemma~\ref{lm:back} are satisfied, which completes the proof.
\begin{figure}[h]
        \centering
        \begin{subfigure}[b]{0.11\textwidth}
                \includegraphics[width=\textwidth]{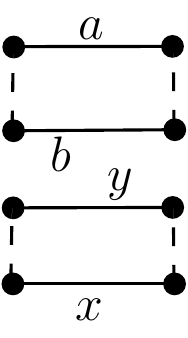}
                \caption{}
                \label{fig:a}
        \end{subfigure}\quad
        \begin{subfigure}[b]{0.11\textwidth}
                \includegraphics[width=\textwidth]{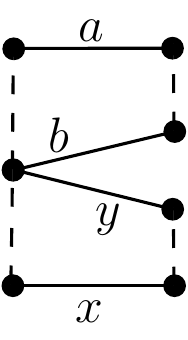}
                \caption{}
                \label{fig:b}
        \end{subfigure}\quad
        \begin{subfigure}[b]{0.11\textwidth}
                \includegraphics[width=\textwidth]{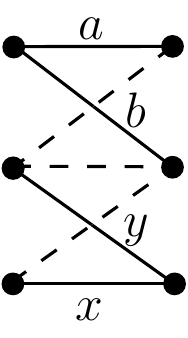}
                \caption{}
                \label{fig:c}
        \end{subfigure}\quad
\begin{subfigure}[b]{0.11\textwidth}
                \includegraphics[width=\textwidth]{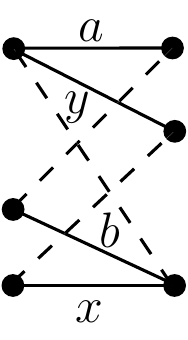}
                \caption{}
                \label{fig:d}
        \end{subfigure}\quad
        \begin{subfigure}[b]{0.11\textwidth}
                \includegraphics[width=\textwidth]{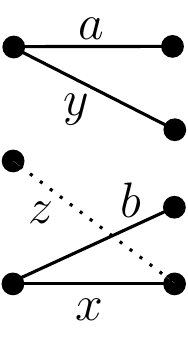}
                \caption{}
                \label{fig:e}
        \end{subfigure}\quad
        \begin{subfigure}[b]{0.11\textwidth}
                \includegraphics[width=\textwidth]{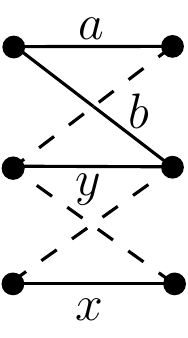}
                \caption{}
                \label{fig:f}
        \end{subfigure}\\
	\begin{subfigure}[b]{0.11\textwidth}
                \includegraphics[width=\textwidth]{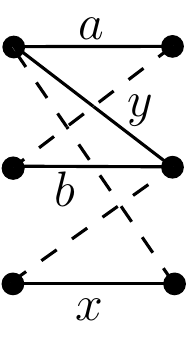}
                \caption{}
                \label{fig:g}
        \end{subfigure}\quad
        \begin{subfigure}[b]{0.11\textwidth}
                \includegraphics[width=\textwidth]{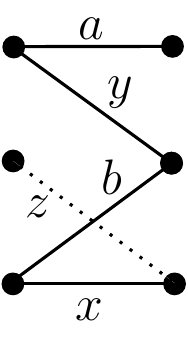}
                \caption{}
                \label{fig:h}
        \end{subfigure}\quad
        \begin{subfigure}[b]{0.11\textwidth}
                \includegraphics[width=\textwidth]{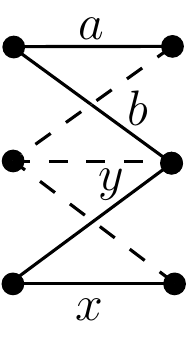}
                \caption{}
                \label{fig:i}
        \end{subfigure}\quad
\begin{subfigure}[b]{0.11\textwidth}
                \includegraphics[width=\textwidth]{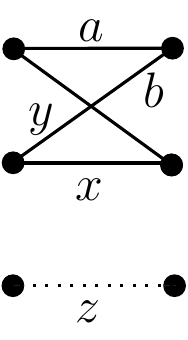}
                \caption{}
                \label{fig:j}
        \end{subfigure}\quad
        \begin{subfigure}[b]{0.11\textwidth}
                \includegraphics[width=\textwidth]{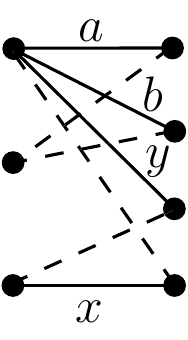}
                \caption{}
                \label{fig:k}
        \end{subfigure}\quad
        \begin{subfigure}[b]{0.11\textwidth}
                \includegraphics[width=\textwidth]{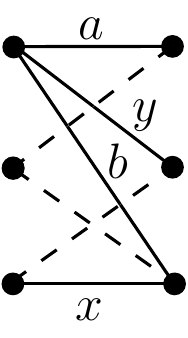}
                \caption{}
                \label{fig:l}
        \end{subfigure}\quad
        \begin{subfigure}[b]{0.11\textwidth}
                \includegraphics[width=\textwidth]{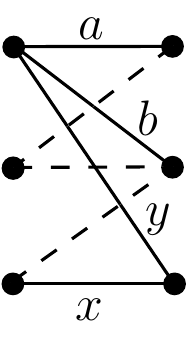}
                \caption{}
                \label{fig:m}
        \end{subfigure}
        \caption{The $a,b,x,y$ configurations}\label{fig:bip}\vspace{-10pt}
\end{figure}

In Figure~\ref{fig:bip} all possible essentially different configurations of incidences between $a$, $b$, $x$ and $y$, up to symmetries, are presented. (The symmetries are defined by exchanging sets $\{a,b\}$ and $\{x,y\}$, and by exchanging elements inside of those two sets.) Configurations in Figure~\ref{fig:bip} are of two types. In the cases where we apply Lemma~\ref{lm:back} with $k=2$, the configurations are extended by (dashed)  edge(s) that form feasible backbone cycles. In other cases we use one auxiliary edge of Lemma~\ref{lm:back} ($k=3$), therefore configurations are extended by the edge $z$ which plays the role of $z_2$ in Lemma~\ref{lm:back}.
\end{proof}

In the previous lemma, linearization characterization only for $\min\{n_1,n_2\}\geq 3$ is given. Note that for configurations  in Figure~\ref{fig:bip}\subref{fig:e}, \ref{fig:bip}\subref{fig:h}, \ref{fig:bip}\subref{fig:i}, \ref{fig:bip}\subref{fig:j}, \ref{fig:bip}\subref{fig:k}, \ref{fig:bip}\subref{fig:e} and \ref{fig:bip}\subref{fig:m}, we actually use the fact that $\min\{n_1,n_2\}\geq 3$. If $\min\{n_1,n_2\}< 3$, linearizable cost matrix $Q$ is not necessary a weak sum matrix. Namely, if $n_1$ or $n_2$ equals to 1, $K_{n_1,n_2}$ is a tree, and if $n_1=n_2=2$, $K_{n_1,n_2}$ is a cycle. In both cases arbitrary $Q$ is linearizable. For the remaining case of $n_1=2$ and $n_2\geq 3$, we present the following counterexample of a symmetric matrix $Q$ that is linearizable but not a weak sum matrix. For $i\neq j$, cost element $q(i,j)$ is equal to 1 if edges $i$ and $j$ are adjacent through an $n_2$-set vertex, and 0 otherwise. Then the linearization costs are given by $c(i)=q(i,i)+2/(n_2+1)$.

\begin{lemma}\label{lm:components}
	Let $Q$ be a linearizable symmetric cost matrix of the QMSTP on a graph $G$. Then for every two distinct biconnected components $I$, $J$ of $G$, submatrix of costs $q(i,j)$, $i\in I$, $j\in J$ is a sum matrix.
\end{lemma}
\begin{proof} If $I$ or $J$ is just one edge, i.e.\@ a bridge, then there is nothing to prove, as every $1\times n$ matrix is a sum matrix. In the rest of the proof we assume that $\min\{|I|,|J|\}\geq 3$.

	We will again make use of backbones. First we fix two edges $a\in I$ and $x\in J$. It is easy to see that for every pair of additional edges $b\in I$ and $y\in J$ there exist an $\{a,b\}$-$\{x,y\}$-backbone. Namely, in every biconnected component, there exist a cycle that contains any pair of edges. Hence, there exist a cycle in $I$ that contains $a$ and $b$, and a cycle in $J$ that contains $x$ and $y$. As their intersection contains at most one vertex, they are feasible backbone cycles. Hence, by Corollary~\ref{cor:back}, the lemma follows.
\end{proof}

Lemma~\ref{lm:sum},~\ref{lm:bipartite},~\ref{lm:components} and Theorem~\ref{thm:cycle}~and~\ref{thm:complete} can be combined to produce the linearization characterization for more general class of graphs.

\begin{theorem}\label{thm:gen}
	Let $G$ be a graph such that its every biconnected component is either a clique, a cycle or a biclique (with vertex partition sets of sizes at least three). Then a symmetric cost matrix $Q$ of the QMSTP on $G$ is linearizable if and only if the submatrices of $Q$ that correspond to different biconnected components are sum matrices, and submatrices that correspond to single biconnected components that are either a clique or a biclique are symmetric weak sum matrices. Furthermore, if $Q$ is linearizable, a linearization can be computed in $O(|E|^2)$ time.
\end{theorem}
\begin{proof}
	Let $Q$ be of the form described in the theorem. We denote by $k$ the number of biconnected components of $G$ that are cycles. Then $Q$ can be expressed as $Q=M+B_1+\cdots +B_k$, where $M$ satisfies the hypothesis of Lemma~\ref{lm:sum}, and $B_i$ is a matrix in which all entries that are not in the submatrix defined by the $i$-th cycle, are equal 0. 
Note that matrices $M$ and $B_i$, $i=1,\ldots,k$, can be found in $O(|E|^2)$ time.
From Lemma~\ref{lm:sum} it follows that $M$ is linearizable and its linearization vector $C_M$ can be computed by \eqref{eq:lin}. From Theorem~\ref{thm:cycle}, it follows that for every $i=1,\ldots,k$, $B_i$ is linearizable and its linearization $C_{B_i}$ is given by \eqref{eq:cycle}. Therefore by Observation~\ref{obs1}, $Q$ is also linearizable and its linearization vector is given by $C=C_M+C_{B_1}+\cdots+C_{B_k}$.

Conversely, if $Q$ is linearizable then it has to be of the form described in the theorem. This follows directly from Lemma~\ref{lm:components} and the proofs of Theorem~\ref{thm:complete} and Lemma~\ref{lm:bipartite}. Namely, backbones of biconnected components can be extended into backbones of $G$ by adding edges that span remaining vertices.
\end{proof}

We present an example that illustrates Theorem~\ref{thm:gen}. Let $G=(V,E)$ be the graph presented by Figure~\ref{fig3}\subref{fig3a}. Graph $G$ has four biconnected components with its corresponding edge sets being $E_1=\{e_1,e_2,e_3\}$, $E_2=\{e_4\}$, $E_3=\{e_5,e_6,e_7,e_8,e_9,e_{10}\}$ and $E_4=\{e_{11}\}$.
\begin{figure}[h]
	 \centering
        \begin{subfigure}[b]{0.25\textwidth}
		\includegraphics[width=\textwidth]{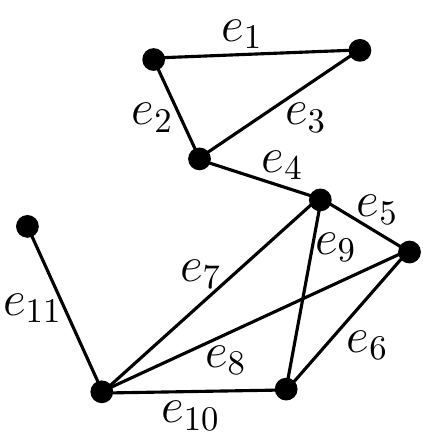}\\ \\
              \caption{}
              \label{fig3a}
        \end{subfigure}\qquad\quad
	 \begin{subfigure}[b]{0.53\textwidth}
              $$\left(\begin{array}{ccc;{2pt/2pt}c;{2pt/2pt}cccccc;{2pt/2pt}c}
			1 & 4 & 8 & 7 & 4 & 6 & 3 & 8 & 5 & 7 & 9 \\
			4 & 2 & 9 & 2 & 3 & 5 & 2 & 7 & 4 & 6 & 9 \\
			8 & 9 & 3 & 4 & 5 & 7 & 4 & 9 & 6 & 8 & 0 \\
\hdashline[2pt/2pt] 7 & 2 & 4 & 8 & 0 & 9 & 3 & 6 & 6 & 3 & 2 \\
\hdashline[2pt/2pt] 4 & 3 & 5 & 0 & 5 & 4 & 5 & 8 & 3 & 7 & 3 \\
			6 & 5 & 7 & 9 & 4 & 2 & 3 & 6 & 1 & 5 & 3 \\
			3 & 2 & 4 & 3 & 5 & 3 & 1 & 7 & 2 & 6 & 4 \\
			8 & 7 & 9 & 6 & 8 & 6 & 7 & 5 & 5 & 9 & 5 \\
			5 & 4 & 6 & 6 & 3 & 1 & 2 & 5 & 8 & 4 & 6 \\
			7 & 6 & 8 & 3 & 7 & 5 & 6 & 9 & 4 & 7 & 0 \\
			\hdashline[2pt/2pt] 9 & 9 & 0 & 2 & 3 & 3 & 4 & 5 &  6 & 0 & 2
		\end{array}\right)$$
              \caption{}
              \label{fig3b}
        \end{subfigure}
	\caption{A linearizable QMSTP instance}
	\label{fig3}\vspace{-10pt}
\end{figure}
Let the symmetric matrix $Q=(q(i,j))$, presented in Figure~\ref{fig3}\subref{fig3b}, be a QMSTP cost matrix associated to $G$, such that $q(i,j)$ is the QMSTP cost associated to the edge pair $(e_i,e_j)$. 
We denote by $Q_{E_iE_j}$ the submatrix of $Q$ consisting of elements $q(k,\ell)$ for $e_k\in E_i$ and $e_{\ell}\in E_j$. In Figure~\ref{fig3}\subref{fig3b}, $Q$ is divided into submatrices $Q_{E_iE_j}$, $i,j\in\{1,\ldots,4\},$ using dashed lines.

Biconnected components $E_1,\ldots,E_4$ are cycles and cliques, so according to Theorem~\ref{thm:gen}, matrix $Q$ is linearizable if and only if submatrices $Q_{E_iE_j}$ have some specific properties. In particular, submatrices that correspond to a pair of different biconnected components, i.e.\@ $Q_{E_iE_j}$ with $i\neq j$, have to be submatrices. There are 12 such submatrices, and 10 of them have one row and/or one column in which case sum matrix property is trivially satisfied. The remaining 2 submatrices are $Q_{E_1E_3}$ and $Q_{E_3E_1}$. Since $Q$ is symmetric, they are transpose of each other, hence it is enough to check sum matrix property  only for one of them. Indeed they are sum matrices, since they are a sum of vectors $(3,2,4)$ and $(1,3,0,5,2,4)$. It remains to check the remaining 4 submatrices $Q_{E_iE_i},$ $i\in\{1,\ldots,4\}$. If $E_i$ is a clique or (big enough) biclique then we need to check whether  $Q_{E_iE_i}$ is a weak sum matrix. If $E_i$ is a cycle then there are no necessary conditions on $Q_{E_iE_i}$. Edges $E_1$ form the complete graph on three vertices, but in the same time $E_1$ forms a cycle. This is not in contradiction, as every symmetric $3\times 3$ matrix is a weak sum matrix. Biconnected components $E_2$ and $E_4$ are trivial cliques, hence the weak sum property of $Q_{E_2E_2}$ and $Q_{E_4E_4}$ is trivially satisfied. $E_3$ is a complete graph, hence it remains to check whether $Q_{E_2E_2}$ is a weak sum matrix. 
It is easy to see that $Q_{E_2E_2}$ is a symmetric weak sum matrix generated by the vector $a=(3,1,2,5,0,4)$, i.e.\@ for $i\neq j$, $i$-th row and $j$-th column of the submatrix $Q_{E_2E_2}$ contains the value $a_i+a_j$. We see that all submatrices satisfy necessary properties, therefore $Q$ is linearizable. 

At this point, it is straightforward to obtain a linearization of $Q$. We can express $Q$ as $Q=M+B_1$, where $M$ is the matrix obtained from $Q$ by replacing elements in the submatrix $Q_{E_1E_1}$ by $0$. Matrix $M$ satisfies Lemma~\ref{lm:sum}, and its linearization vector $C_M$ can be calculated as described in the proof of Lemma~\ref{lm:sum} using vectors obtained in the analysis above. Furthermore, $B_1$ is linearizable and its linearization vector $C_B$ can be calculated as described in Theorem~\ref{thm:cycle}. Then vector $C=C_M+C_B$ is a linearization of $Q$. And $C=(54,41,48,12,23,42,23,67,40,45,2)$ is one such vector in the case of matrix $Q$.

\section{Recognition of Linearizable QMSTP}

Theorem~\ref{thm:gen} gives us a solution for the quadratic spanning tree linearization problem (QST-LP) for the class of graphs  in which every biconnected component is either a clique, a biclique or a cycle. Given such graph $G=(V,E)$, one can find in linear time its biconnected components (see \cite{ht73}), and determine which type they are. 
Now for a given (not necessary symmetric) cost matrix $Q$, from Observation~\ref{obs3} if follows that $Q$ is linearizable if and only if the symmetric matrix $\frac{1}{2}(Q+Q^T)$ is linearizable. According to Theorem~\ref{thm:gen}, to determine whether $Q$ is linearizable we need to check whether appropriate submatrices of $\frac{1}{2}(Q+Q^T)$ are sum matrices or symmetric weak sum matrices. 
In worst case this takes $\Theta(|E|^2)$ time, since potentially every element of $Q$ which is not on the main diagonal has to be examined. Next we examine whether the recognition can be done faster if the cost matrix is given in the factored form.

Let $M=(m(i,j))$ be an $n\times n$ matrix of rank $p$. Then the elements of $M$ are of the form
\begin{equation}\label{eq:fact}
	M(i,j)=\sum_{k=1}^p a_i^kb_j^k,
\end{equation}
for some vectors $a^k$ and $b^k$, $k=1,\ldots,p$. Hence, an $n\times n$ matrix of a rank $p$ can be represented with $2pn$ values. We say that \eqref{eq:fact} is a \emph{factored form representation} of matrix $M$.

Note that every sum matrix can be written as the sum of a constant row matrix and a constant column matrix. Since $\text{rank}(M_1+M_2)\leq \text{rank}(M_1)+\text{rank}(M_2)$ for every matrices $M_1$ and $M_2$, it follows that every sum matrix has the rank at most 2. Therefore, the problem of recognizing sum matrices represented in a factored form \eqref{eq:fact} is reduced to the following question.
\emph{Given $a_i,b_i,c_i,d_i$, $i=1,\ldots,n$, is it possible to decide in $O(n)$ time whether the matrix $M=(m(i,j))$ with  $m(i,j)=a_ib_j+c_id_j$ is a sum matrix?}
An affirmative answer to this question follows from the following theorem.

\begin{theorem}\label{thm:rank2sum}
	Let an $n\times n$ matrix $M=(m(i,j))$ be of the form $m(i,j)=a_ib_j+c_id_j$, $i,j=1,\ldots,n$.
	\begin{itemize}
		\item If at least one of the vectors $a,b,c,d$ is a constant vector, then $M$ is a sum matrix if and only if $a$ or $b$ is a constant vector, and $c$ or $d$ is a constant vector.
		\item If none of the vectors $a,b,c,d$ is a constant vector, then $M$ is a sum matrix if and only if there exist three constants $K\neq 0$, $K_1$ and $K_2$ such that $a_i=Kc_i+K_1$ and $d_i=-Kb_i+K_2$, $i=1,\ldots,n$.
	\end{itemize}
\end{theorem}

\begin{proof}
	Let a matrix $M=(m(i,j))$ be of the form $m(i,j)=a_ib_j+c_id_j$. Let us assume $M$ is a sum matrix, i.e.\@ there exist two vectors $e$ and $f$ such that $m(i,j)=e_i+f_j$, $i,j=1,\ldots,n$. Then for arbitrary $i,j,k,\ell\in\{1,\ldots,n\}$
\[
	m(i,k)-m(i,\ell)=f_k-f_\ell \quad \text{and} \quad m(j,k)-m(j,\ell)=f_k-f_\ell.
\]
Hence $m(i,k)-m(i,\ell)=m(j,k)-m(j,\ell)$. Now from $m(i,j)=a_ib_j+c_id_j$ it follows that
\[
	a_ib_k+c_id_k-a_ib_\ell-c_id_\ell=a_jb_k+c_jd_k-a_jb_\ell-c_jd_\ell,
\]
which can be rearranged to
\[
	a_i(b_k-b_\ell)+c_i(d_k-d_\ell)=a_j(b_k-b_\ell)+c_j(d_k-d_\ell).
\]
Finally, we get a necessary condition that if $M$ is a sum matrix then
\begin{equation}\label{eq6}
	(a_i-a_j)(b_k-b_\ell)=-(c_i-c_j)(d_k-d_\ell),
\end{equation}
for every $i,j,k,\ell\in \{1,\ldots,n\}$.
Now we divide our investigation into two cases. 

\textbf{Case 1}: \textit{At least one of the vectors $a,b,c,d$ is a constant vector.} Without loss of generality we can assume that $a$ is a constant vector. From \eqref{eq6} it follows that
\begin{equation}\label{eq7}
	(c_i-c_j)(d_k-d_\ell)=0\quad \forall i,j,k,\ell\in\{1,\ldots,n\}.
\end{equation}
Hence either $c$ or $d$ is a constant vector. Otherwise there would exist $i,j$ for which $c_i-c_j\neq 0$, and $k,\ell$ for which $d_k-d_\ell\neq 0$ which would contradict \eqref{eq7}.

Note that this is also a sufficient condition. Let us assume $a_i=\alpha$ and $d_i=\delta$, $i=1,\ldots,n$. Then,
\[
	m(i,j)=\alpha b_j+c_i\delta=e_i+f_j \quad \forall i,j\in\{1,\ldots,n\},
\]
where $e_i:=\delta c_i$ and $f_i:=\alpha b_i$, $i=1,\ldots,n$. In the case when $c$ (instead of $d$) is a constant vector, in a similar way one gets that $M$ is a sum matrix.

\textbf{Case 2}: \textit{None of the vectors $a,b,c,d$ is a constant vector.} Assume that there are two elements of vector $a$ that are the same, i.e.\@ there exist $i,j$, $i\neq j$ such that $a_i=a_j$. Then for the same $i,j$ $c_i=c_j$ holds. Assume the contrary, i.e.\@  $a_i=a_j$ and $c_i\neq c_j$. As $d$ is not a constant vector, there exist $k,\ell$ such that $d_k\neq d_\ell$. Now for such $i,j,k,\ell$, equation \eqref{eq6} does not hold, which is a contradiction. Hence, $a_i=a_j$ if and only if $c_i=c_j$. Using the same logic, for all $i,j\in\{1,\ldots,n\}$, $b_i=b_j$ if and only if $d_i=d_j$.

Let $N_1\subseteq \{1,\ldots,n\}$ be a maximal set of indices $i$ for which $a_i$'s (and $c_i$'s) are pairwise distinct. That is, for every $i,j\in N_1$, $i\neq j$, it follows that $a_i\neq a_j$ (and hence $c_i\neq c_j$ also). Let $N_2$ be a set of indices with the same property for vectors $b$ and $d$. Now from \eqref{eq6} it follows that
\[
 \frac{a_i-a_j}{c_i-c_j}=-\frac{d_k-d_\ell}{b_k-b_\ell},
\]
for every distinct $i,j\in N_1$ and $k,\ell\in N_2$. By fixing some distinct $k,\ell\in N_2$, it follows that $(a_i-a_j)/(c_i-c_j)$ is a nonzero constant (which we denote by $K$) for every distinct $i,j\in N_1$. Analogously, it follows that $(d_k-d_\ell)/(b_k-b_\ell)=-K$ for every distinct $k,\ell\in N_2$. Hence
\[
	a_i-a_j=K(c_i-c_j)=Kc_i-Kc_j\quad \forall i,j\in N_1,
\]
from which it follows that $a_i=Kc_i+K_1$ for some constant $K_1$ and for all $i\in N_1$. Note that from the way we defined $N_1$, this relation can be extended to entire $\{1,\ldots, n\}$, i.e.\@ we have that
\begin{equation}\label{eq8}
	a_i=Kc_i+K_1 \quad i=1,\ldots, n,
\end{equation}
for some constants $K\neq 0$ and $K_1$. Analogously we obtain that
\begin{equation}\label{eq9}
	d_i=-Kb_i+K_2 \quad i=1,\ldots, n,
\end{equation}
for some additional constant $K_2$.

Note that \eqref{eq8} and \eqref{eq9} are sufficient conditions also. Namely
\begin{align*}
	m(i,j)&=a_ib_j+c_id_j\\
		&=(Kc_i+K_1)b_j+c_i(-Kb_j+K_2)\\
		&=K_2c_i+K_1b_j
\end{align*}
is a sum matrix relation.
\end{proof}

\begin{corollary}\label{cor:rec}
	Given $a_i,b_i,c_i,d_i$, $i=1,\ldots,n$, it is possible to decide in $O(n)$ time whether the square matrix $M=(m(i,j))$ with  $m(i,j)=a_ib_j+c_id_j$ is a sum matrix.
\end{corollary}
\begin{proof}
	It follows directly form the statement and the proof of Theorem~\ref{thm:rank2sum}. Namely, the following is an $O(n)$ time algorithm. 

First we check whether any of the vectors $a,b,c,d$ are a constant vectors. If so, then $M$ is a sum matrix if and only if $a$ or $b$ is a constant vector, and $c$ or $d$ is a constant vector. Else, find $i$ and $j$ such that $a_i-a_j\neq 0$ and define $K$ to be $K=(a_i-a_j)/(c_i-c_j)$. Furthermore, define $K_1,K_2$ to be $K_1=a_1-Kc_1$ and $K_2=d_1+Kb_1$. Then $M$ is a sum matrix if and only if \eqref{eq8} and \eqref{eq9} are satisfied.
\end{proof}

In the case of complete graphs, the following  result on the recognition of linearizable cost matrices represented in factored form straightforwardly holds.

\begin{corollary}
	Let $G=(V,E)$ be a complete graph or a complete bipartite graph. Let $Q=(q(i,j))$ be a symmetric cost matrix of a QMSTP on the graph $G$ such that $q(i,j)=a_ib_j+c_id_j$, $i,j=1,\ldots,|E|,$ $i\neq j,$ for some given vectors $a,b,c,d$. Then in $O(|E|)$ time it can be decided whether $Q$ is linearizable, and if so, a linearization can be calculated in $O(|E|)$ time.
\end{corollary}
\begin{proof}
	If follows directly form Corollary~\ref{cor:rec} and the fact that in the case of complete (bipartite) graphs, $r(i)$ from \eqref{eq:lin} can be calculated in $O(1)$ time for every $i=1,\ldots,|E|$.
\end{proof}

\section{Conclusions and Future Work}

We investigated the problem of characterizing linearizable QMSTP cost matrices, and we resolved the problem for a broad class of graphs. The main result is presented as Theorem~\ref{thm:gen}. In particular, given a graph $G$, Lemma~\ref{lm:sum} gives a sufficient condition for a cost matrix to be linearizable, and in the case of complete and complete bipartite graphs, the condition is also necessary. A natural question that imposes itself is: \emph{For which graphs the conditions of Lemma~\ref{lm:sum} is necessary?} In the view of Lemma~\ref{lm:components}, this question can be rephrased as the following open problem: \emph{For which biconnected graphs a symmetric QMSTP cost matrix is linearizable only if it is a weak sum matrix.}

In this paper so far we have encountered two types of biconnected graphs for which a linearizable QMSTP cost matrix does not need to be a weak sum matrix. These graphs were cycles and  complete bipartite graphs $K_{2,n}$. Note that both of these graphs classes contain a vertex with degree $2$. As a matter of fact, for every biconnected graph that contains a vertex with degree $2$, the weak sum condition is not necessary. Namely, let $G=(V,E)$ be a biconnected graph such that $p\in V$ is of degree $2$ and $E_p$ is the set of two edges adjacent to $p$. Then the following symmetric matrix $Q=(q(i,j))$ given by
$$q(e,f) = \begin{cases} 
	1/2 & \mbox{if } e,f\in E_p,\ e\neq f,\\ 
	1/(2(n-3)) & \mbox{if } e,f\in E\setminus E_p,\ e\neq f,\\ 
	0 & \mbox{otherwise,}
\end{cases}$$
is linearizable, but it is not a weak sum matrix. The linearization is given by $c(e)=\frac{n-3}{n-1}$, $e\in E$. (Note that such cost matrices have even stronger property that the cost of every spanning tree is the same.)
Therefore, an interesting question would be to identify how dense a graph needs to be in order that linearizable cost matrices are necessarily weak sum. Is it enough that the minimum vertex degree is at least $3$?

\section*{Acknowledgment}

This research work was supported by an NSERC discovery grant and an NSERC discovery accelerator supplement awarded to Abraham P. Punnen.

\end{document}